\documentclass[11pt,a4paper,reqno]{amsart}
\usepackage[hmargin=3.5cm,vmargin=3cm]{geometry}
\usepackage{amsmath}
\usepackage{amssymb}
\usepackage{amsthm}
\usepackage[mathscr]{eucal}
\usepackage{enumitem}
\usepackage{calrsfs}  

\newtheorem{theorem}{Theorem}[section]
\newtheorem{lemma}[theorem]{Lemma}
\newtheorem{prop}[theorem]{Proposition}
\newtheorem{corollary}[theorem]{Corollary}
\theoremstyle{definition}

\newtheorem{example}[theorem]{Example}
\newtheorem{remark}[theorem]{Remark}
\numberwithin{equation}{section}
\swapnumbers
\numberwithin{equation}{section}

\newcommand{\R}{\mathbb{R}}

\renewcommand{\d}{\mathrm{d}}

\renewcommand{\epsilon}{\varepsilon}

\renewcommand{\P}{\mathbb{P}}
\newcommand{\E}{\mathbb{E}}

\newcommand{\1}{1}
\newcommand{\F}{\mathcal{F}}

\newcommand{\RR}{\mathcal{R}}
\newcommand{\I}{\mathcal{I}}

\newcommand{\KK}{\mathcal{K}}

\newcommand{\LL}{\mathcal{L}}
\renewcommand{\SS}{\mathcal{S}}

\newcommand{\UU}{\mathcal{U}} 
\newcommand{\HH}{\mathrm{H}} 
\begin{document}





 
\title{Representation of Self-similar Gaussian Processes}
\date{\today}

\author[Yazigi, A.]{Adil Yazigi}
\address{Adil Yazigi\\
Department of Mathematics and Statistics\\ 
University of Vaasa
P.O.Box 700\\ 
FIN-65101 Vaasa\\ 
Finland}
\email{adil.yazigi@uwasa.fi}

\begin{abstract} We develop the canonical Volterra representation for  a self-similar   Gaussian process by using the Lamperti transformation of the corresponding stationary Gaussian process, where this latter one admits a canonical integral representation under the assumption of pure non-determinism. We apply the representation obtained to the equivalence in law for self-similar Gaussian processes.
\end{abstract}

\thanks{The author wishes to thank Tommi Sottinen for discussions and helpful comments. The author also thanks  the Finnish Doctoral Programme in Stochastics and Statistics and  the Finnish Cultural Foundation for financial support.}
\maketitle

{\small

}


\section{Introduction and preliminaries}

In this paper, we will construct the canonical Volterra representation for a given self-similar centered Gaussian processes. The role of the canonical Volterra representation  which was first introduced by Levy  in \cite{Lev1} and \cite{Lev2}, and later developed by Hida in \cite{hida}, is to provide an integral representation for a Gaussian process $X$  in terms of a Brownian motion $W$ and a non-random Volterra kernel  $k$ such that the expression
$$X_t=\int_0^t k(t,s)\, \d W_s$$
holds for all $t$ and the Gaussian processes $X$ and $W$ generate the same filtration. It is known, see \cite{Cel} and \cite{Lev1}, that if  the kernel $k$ satisfies the homogeneity property for some degree $\alpha$, i.e. $k(at,as)=a^{\alpha}k(t,s)$, $a>0$, the Gaussian process $X$ is self-similar with index $\alpha+ \frac 12$. Thus, the main goal of this paper is to give, under some suitable conditions, a general construction of the canonical Volterra representation  for  self-similar Gaussian processes, and  which also guaranties the homogeneity property of the kernel. In section 2, the linear Lamperti transform that defines the one-one correspondence between stationary processes  and self-similar processes, will be used to express the explicit form of the canonical Volterra representation for self-similar Gaussian processes in the light of the classical canonical representation of the stationary processes given by Karhunen in \cite{Karhunen}. In section 3, we give an application of the representation obtained to  a Gaussian process equivalent in law to the self-similar Gaussian process.

In our mathematical settings, we take $T>1$ to be a fixed time horizon, and  on a complete probability space $(\Omega, \F, \P)$ we consider a centered Gaussian process $X=(X_t;  t\in\left[ 0,T \right])$ that enjoys the  self-similarity property for some $\beta >0$, i.e.
$$
(X_{at})_{0 \leq t \leq T/a}\overset{d}{=} (a^\beta X_t)_{0 \leq t \leq T},\quad  \mbox{for all}\; a>0,
$$
where $\overset{d}{=}$ denotes equality in distributions, or equivalently ,
\begin{equation}\label{cov1}
r(t,s)=\E (X_tX_s)=T^{2\beta}\, r\left(\frac{t}{T},\frac{s}{T}\right), \quad 0\leq t, s \leq T.
\end{equation}
In particular, we have $r(t,t)=t^{2\beta}\E(X_1^2)$, which is finite and continuous function at every $(t,t)$ in $[0,T]^2$, and therefore, is continuous at every $(t,s)\in [0,T]^2$, see \cite{loeve}. A consequence of the continuity of the covariance function $r$ is that $X$ is mean-continuous. 

We denote by $\HH_X(t)$  the  closed linear subspace of  $L^2([0,T])$  generated by Gaussian random variables $X_s$ for $s \leq t$, and by $(\F_t^X)_{t\in[0T]}$, where $\F_t^X:=\sigma (X_s, s\leq t)$, the completed natural filtration of $X$. We call the \emph{Volterra representation} of $X$ the integral representation of the form 
\begin{equation}\label{volrep}
X_t= \int_0^t k(t,s)\,d W_s,\quad t \in [0,T],
\end{equation}
where $W=(W_t;  t\in\left[ 0,T \right])$ is a standard Brownian motion and  the kernel $k(t,s)$ is a Volterra kernel, i.e. a measurable function on $[0,T]\times[0,T]$ that satisfies  $\int_0^T \int_0^t k(t,s)^2\,\d s \, \d t< \infty$, and  $k(t,s)=0$ for $s>t$. The Gaussian process $X$ with such representation is called a \emph{Gaussian Volterra process}, provided with $k$ and $W$.

Moreover, the Volterra representation is said to be \emph{canonical} if the \emph{canonical property} 
$$\F_t^X=\F_t^W$$ 
holds for all $t$, or equivalently 
\begin{equation}\label{cano}
\HH_X(t)= \HH_W(t), \quad \mbox{for all} \;\, t.
\end{equation}

\begin{remark}\label{remark}
\begin{enumerate}
\item An equivalent to the canonical property is that if  there exists a random variable  $\eta=\int_0^T \phi(s)\, \d W_s$ , $\phi \in L^2([0,T])$, such that it  is independent of $X_t$ for all $0\leq t \leq T$, i.e. $\int_0^t k(t,s)\, \phi(s)\, \d s=0$ , one has $\phi\equiv0$. This means that the family $\left\{k(t,\cdot), 0\leq t \leq T\right\}$ is free and spans a vector space that is dense in  $L^2([0,T])$. If we associate with the canonical kernel $k$ a Volterra integral operator $\KK$ defined on $ L^2([0,T])$ by $\KK\phi(t)=\int_0^t k(t,s)\, \phi(s)\, \d s $, it follows from the canonical property \eqref{cano} that $\KK$ is injective  and $\KK(L^2([0,T]))$ is dense in $L^2([0,T])$. The covariance integral operator $\RR$ associated with the kernel $r(t,s)$ has the decomposition $\RR=\KK \KK^*$, where $\KK^*$ is the adjoint operator of $\KK$. In this case, the covariance $r$  is factorable, i.e. 
$$r(t,s)=\int_0^{t\wedge s} k(t,u)k(s,u)\, \d u.$$
\item  A special property for a Volterra integral operator is that  it has no eigenvalues, see \cite{gohberg}.
\end{enumerate}
\end{remark}

\section{The Canonical Volterra representation and self-similarity}

The Gaussian process $X$ is $\beta$--self-similar, and according to  Lamperti \cite{Lamp}, it can be transformed into a  stationary Gaussian  process $Y$ defined by:
\begin{equation}\label{stat}
Y(t):=e^{-\beta t} X(e^t), \quad t \in ( -\infty, \log T].
\end{equation}
Conversely, $X$ can be recovered from $Y$ by the inverse Lamperti transformation
\begin{equation}
X(t)=t^{\beta}Y(\log t),\quad t \in [0,T].
\end{equation}
It is obvious that the mean-continuity of the process $Y$ follows from the fact that
$$\E (Y_t-Y_s)^2=2 \left(r(1,1)-e^{-(t-s)\beta} r(e^{t-s},1)\right)$$
converges to zero when $t$ approaches $s$. As was shown by Hida \& Hitsuda (\S 3, \cite{HidaHitsuda}), which is a well-known classical result that has been first established by Karhunen (\S 3, Satz 5, \cite{Karhunen}), the stationary Gaussian process $Y$ admits the canonical representation
\begin{equation}\label{statcano}
Y_t= \int_{-\infty}^t G_T(t-s) \, \d W_s^*,
\end{equation}
where  $G_T$ is a measurable function that belongs to $L^2(\R, \d u)$ such that  $G_T(u)=0$ when $u<0$, and $W^*$ is a standard Brownian motion satisfying  the canonical property, i.e.,  $\HH_Y(t)=\HH_{W^*}(t)$, $t \in [0,T]$. A necessary and sufficient condition for the existence of the representation \eqref{statcano} is that $Y$ is purely non-deterministic.  Following Cramer  \cite{cramer},  a  process $Z$ is  purely non-deterministic if and only if the condition
\begin{equation}\label{condition}
\bigcap_{t} \HH_{Z}(t)=\{0\} \tag{C},
\end{equation}
 is fulfilled, where $\{0\}$ is the $L^2$--subspace spanned by the constants. The condition  \eqref{condition} means that  the remote past is trivial, i.e. $\F_{0^+}^Z$ is trivial; see also  \cite{Hida},  \cite{HidaHitsuda} and \cite{Karhunen}. 

Next, we shall extend the property of pure non-determinism  to the  self-similar centered Gaussian process $X$.

\begin{theorem}\label{GVP} 
The self-similar centered Gaussian process $X=(X_t;  t\in\left[ 0,T \right])$  satisfies the condition \eqref{condition} if and only if there exist a standard Brownian motion $W$ and a Volterra kernel  $k$ such that $X$ has the representation
\begin{equation}\label{GVPrep}
X_t= \int_0^t k(t,s)\, \d W_s,
\end{equation}
where the Volterra kernel $k$ is defined by
\begin{equation}\label{canokernel}
k(t,s)= t^{\beta-\frac 12}\, F\left(\frac{s}{t}\right), \quad s<t,
\end{equation}
 for some function $F \in L^2(\R_+,\d u)$ independent of $\beta$, with $F(u)=0$ for $1<u$.

Moreover,  $\HH_X(t)= \HH_W(t)$ holds for each $t$.
\end{theorem}

\begin{remark}\label{notpnd}
In the case where the process $X$ is  \emph{trivial self-similar}, i.e. $X_t=t^{\beta}W_1$, $ 0\leq t \leq T$, the condition \eqref{condition} is not satisfied since  $\bigcap_{t\in (0,T)} \HH_{X}(t)=\HH_W(1)$. Thus, $X$ has no Volterra representation in this case.
\end{remark}

\begin{proof}  The fact that $X$ is purely non-deterministic is equivalent to that $Y$ is purely non-deterministic since
$$\bigcap_{t\in (0,T)} \HH_{X}(t)=\bigcap_{t\in (0, T)} \HH_{Y}(\log t)=\bigcap_{t\in (-\infty, \log T)} \HH_{Y}(t).$$
Thus $Y$ admits the representation \eqref{statcano} for some square integrable kernel $G_T$ and a standard Brownian motion $W^*$.  By the inverse Lamperti transformation, we obtain 
$$
X(t)=\int_{-\infty}^{\log t} t^{\beta} G_T(\log t- s) \, \d W_s^*=\int_0^t t^{\beta} s^{- \frac 12} G_T\left(\log \frac ts\right) \, \d W_s,
$$
where $\d W_s= s^{ \frac 12}\d W^*_{\log s}$. We take the Volterra kernel $k$ to be defined as  $k(t,s)= t^{\beta-\frac 12}\, F\left(\frac{s}{t}\right)$, where  $F(u)=u^{-\frac 12}G_T(\log u^{-1}) \in L^2(\R_+, \d u)$ vanishing when $u<1$ since $G_T(u)=0$ when $u<0$, i.e. for $t<s$, we have  $F(\frac{s}{t})=0$, and then, $k(t,s)=0$. Indeed,
\begin{eqnarray*} 
\int _0^{\infty}  F(u)^2 \, \d u =\int _0^{\infty}  G_T(\log u^{-1})^2 \, \frac{\d u}{u}=\int_{-\infty}^{\infty} G_T(v)^2\, \d v< \infty,
\end{eqnarray*}
and
\begin{eqnarray*} 
\int _0^T \int_0^t  F\left(\frac st\right)^2 \,\d s\, \d t &=&\int _0^T t \, \d t  \int _0^1  F(u)^2 \, \d u \\
&=& \int _0^T t \, \d t \int_0^{\infty} G_T(v)^2\, \d v< \infty.
\end{eqnarray*}
Thus,
\begin{eqnarray*} 
\int_0^T \int_0^t t^{2\beta-1} F\left(\frac{s}{t}\right)^2 \, \d s\, \d t
&=& \left(\int_0^T t^{2\beta}\right) \left(\int _0^1  F\left(u\right)^2 \, \d u\right)\, \d t <\infty
\end{eqnarray*}
Considering the closed linear subspace $\HH_{\d W}(t)$ of  $L^2([0,T])$ that is generated by $W_s-W_u$ for all $u\leq s\leq t$, we have $\HH_{\d W}(t)=\HH_W(t)$ since $W_0=0$, and therefore, the canonical property follows from the equalities
$$ \HH_X(t)=\HH_Y(\log t)=\HH_{\d W^*}(\log t)=\HH_{\d W}(t)=\HH_W(t).$$
\end{proof}

\begin{example}[Fractional Brownian motion]
The fractional Brownian motion (fBm) on $[0,T]$ with index $H\in (0,1)$ is a centered Gaussian process $B^H = (B_t; 0\leq t\leq T)$ with the covariance function $R_H(t,s)=\frac 12(s^{2H} + t^{2H} - |t - s|^{2H} )$. The fBm is  $H$-self--similar, and following \cite{alios} and \cite{decreuse}, it admits the canonical Volterra representation with the canonical kernel 
\begin{eqnarray*}
k_H(t,s)&=& c_H  s^{\frac 12 -H} \int_s^t (u- s)^{H- \frac 32}\, u^{H- \frac 12}\, \d u, \quad \mbox{for}\; H> \frac 12,\\
k_H(t,s)&=& d_H \Biggl( \left(\frac ts\right)^{H-\frac 12}(t-s)^{H-\frac 12}\\
&& - \left(H- \frac 12\right) s^{\frac 12 -H}\int_s^t u^{H- \frac 32}\, (u- s)^{H- \frac 12}\, \d u \Biggr),
\quad \mbox{for}\; H< \frac 12, 
\end{eqnarray*}
where $c_H=\left(\frac{H(2H-1)}{B(2-2H,H-\frac 12)}\right)^{\frac 12}$,  $d_H=\left(\frac{2H}{(1-2H)B(1-2H,H+\frac 12)}\right)^{\frac 12}$, here $B$ denotes the Beta function. So, the function $F$ that corresponds to the canonical Volterra representation of fBm has the expressions:
\begin{equation*}
F(u)=c_H\left(u^{\frac 12 -H} \int_u^1 (z- u)^{H- \frac 32}\, z^{H- \frac 12}\, \d z\right), \quad \mbox{for}\; H> \frac 12,
\end{equation*}
and
\begin{equation*}
F(u)=d_H \left(\left(\frac 1u-1\right)^{H-\frac 12} - \left(H- \frac 12\right)\left(u\right)^{\frac 12 -H}\int_u^1 z^{H- \frac 32}\, (z- u)^{H- \frac 12}\, \d z\right),
\end{equation*}
for $H< \frac 12$.
\end{example}

A function $f(t,s)$ is said to be homogeneous with degree $\alpha$ if  the equality 
$$
f(at,as)=a^{\alpha}f(t,s),\quad a>0,
$$
 holds  for all $t,s$ in $[0,T]$. From the expression \eqref{canokernel} of the canonical kernel, it is easy to see that $k$ is homogeneous  with degree $\beta- \frac 12$, i.e. $k(t,s)=T^{\beta-\frac12}k(\frac tT, \frac sT)$, \,for all $s<t \in [0,T]$.

Given $X$ with the canonical Volterra representation \eqref{GVPrep}, let $\UU$ to be a bounded unitary endomorphism on $ L^2([0,T])$ with adjoint $\UU^*=\UU^{-1}$, and define the process $B=(B)_t:=(\UU^*(W))_t$ for each $t\in [0,T]$. Indeed, $B$ is a standard Brownian motion since the Gaussian measure is preserved under the unitary transformations. With the notation $k_t(\cdot):=k(t,\cdot)$, the Gaussian process associated with the kernel $(\UU k_t)(s)$ and the standard Brownian motion $B$ has same law as $X$. For the covariance operator, we write 
$$\RR=\KK \KK^*=\KK \UU^* \UU \KK^*=(\KK \UU^*)(\KK \UU^*)^*,$$
where the operator 
$\KK \UU^*$ is defined by 
$$(\KK \UU^*)\phi(t)=\int_0^t k(t,s)\, (\UU^* \phi)(s)\, \d s=\int_0^T(\UU k_t)(s) \, \phi(s)\, \d s, \quad \phi \in  L^2([0,T]).$$ 
The associated Gaussian process has then the integral representation $\int_0^T(\UU k_t)(s) \, \d B_s $ for all $t \in [0,T]$.
\begin{corollary}
For any  bounded unitary endomorphism $\UU$ on $ L^2([0,T])$,  the homogeneity of $k$ is preserved under $\UU$.
\end{corollary}

\begin{proof}
Let $\UU$ be a bounded unitary endomorphism  on $ L^2([0,T])$, and let the scaling operator $\SS f(t)= T^{\frac 12}f(Tt)$ with adjoint $\SS^*f(t)=T^{-\frac 12}f(\frac tT)$ to be defined for all $f \in L^2([0,T])$. The homogeneity of $k$ means that
$$k_t(s)=T^{\beta}(\SS^*k_{\frac tT})(s),$$
then we have
$$
\UU k_t(s)=T^{\beta}(\UU \SS^*k_{\frac tT})(s)=T^{\beta-\frac 12}(\SS \UU \SS^* k_{\frac tT})(\frac sT).
$$
To show the equality $\SS \UU \SS^* k_{\frac tT}=\UU  k_{\frac tT}$, we will use the Mellin transform 
\begin{eqnarray*}
\int_0^{\infty} (\SS \UU \SS^* k_{\frac tT})(s)\, s^{p-1}\, \d s &=&\int_0^{\infty}  (\UU \SS^* k_{\frac tT})(s)\, (\SS^* s^{p-1})\, \d s\\
&=& T^{ \frac 12-p}\int_0^{\infty}  (\UU \SS^* k_{\frac tT})(s)\,  s^{p-1}\, \d s\\
&=& T^{ \frac 12-p}\int_0^{\infty}  ( \SS^* k_{\frac tT})(s)\, (\UU^* s^{p-1})\, \d s\\
&=& T^{-p}\int_0^{\infty}  k_{\frac tT}(\frac sT) \, (\UU^*s^{p-1}) \, \d s \\
&=&  \int_0^{\infty}  k_{\frac tT}(u) \, (\UU^*u^{p-1}) \, \d u = \int_0^{\infty} \UU k_{\frac tT}(u)\,  u^{p-1} \, \d u,\\
\end{eqnarray*}
and the uniqueness property of the Mellin transform implies that $$\SS \UU \SS^* k_{\frac tT}=\UU  k_{\frac tT}.$$
\end{proof}

\begin{remark}
The fact that the  $\beta$-self--similar Gaussian process $X$ satisfies the condition \eqref{condition}, guaranties the existence of the canonical kernel $k$ which is homogeneous with degree $\beta- \frac 12$, and its homogeneity is preserved under unitary transformation. If we consider again the example  in Remark \ref{notpnd}, one has the representation 
$$X_t=\int_0^T t^{\beta} \1_{[0,1]}(s)\, \d W_s, 0 \leq t \leq T,$$
where $\1_{[0,1]}(s)$ is the indicator function. In this case, we see that the kernel $t^{\beta} \1_{[0,1]}(s)$ does not satisfy the homogeneity property of any degree.
\end{remark}


\section{Application to the equivalence in law}

In this section, we shall  emphasize the self-similarity property under the equivalence of laws of Gaussian processes. First, We recall the results shown by Hida-Hitsuda in the case of  Brownian motion, see \cite{HidaHitsuda} and \cite{Hitsuda}. Following  Hitsuda's representation theorem, a centered  Gaussian process $\widetilde{W}=(\widetilde{W}_t;  t\in\left[ 0,T \right])$ is  equivalent in law to a standard Brownian motion $W=(W_t;  t\in\left[ 0,T \right])$ if and only if  $\widetilde{W}$ can be represented in a unique way by 
\begin{equation}\label{canodecom}
 \widetilde{W}_t= W_t -\int_0^t \int_0^s l(s,u)\, \d W_u\, \d s,
\end{equation}
where $l(s,u)$ is a Volterra kernel, i.e. 
\begin{equation}\label{cond}
\int_0^T \int_0^t l(t,s)^2\, \d s\,\d t<\infty,\qquad l(t,s)=0\quad \mbox{for}\quad t<s,
\end{equation}
and such that the equality $\HH_{\widetilde{W}}(t)=\HH_{W}(t)$ holds for each $t$. We note here that the uniqueness of the canonical decomposition    \eqref{canodecom} is  in the sense that if  $l'$ is a Volterra kernel  and $W'=(W'_t;  t\in\left[ 0,T \right])$ is a standard Brownian motion such that for $0 \leq t \leq T$
$$W'_t -\int_0^t \int_0^s l'(s,u)\, \d W'_u\, \d s=W_t -\int_0^t \int_0^s l(s,u)\, \d W_u\, \d s,$$
then $l=l'$ and $W=W'$.

If we denote by $\P$ and $\widetilde{\P}$ the laws of $W$ and $\widetilde{W}$ respectively,  these two processes are equivalent in law if $\P$ and $\widetilde{\P}$ are equivalent, and the Radon-Nikodym density is given by
$$\frac{\d \widetilde{\P}}{\d \P}= \exp\left\{\int_0^T\int_0^s l(s,u)\d W_u\,\d W_s-\frac 12 \int_0^T\left(\int_0^s l(s,u)\d W_s \right)^2\d s\right\}.$$
 The centered Gaussian process $\widetilde{W}$ is  a standard Brownian motion under $\widetilde{\P}$ with $\widetilde{\E}(\widetilde{W}_t\widetilde{W}_s)=\E(W_tW_s)$,  hence, it is self-similar with index $\frac 12$ under $\widetilde{\P}$. It follows from \eqref{canodecom} that the covariance of $\widetilde{W}$ under $\P$ has the form of
\begin{eqnarray*}
\E(\widetilde{W}_t\widetilde{W}_s)&=& t\wedge s- \int_0^{t \wedge s} \int_u^s l(v,u) \, \d v\, \d u - \int_0^{t \wedge s} \int_u^t l(v,u) \, \d v\, \d u\\
&+& \int_0^t\int_0^s \int_0^{v_1\wedge v_2} l(v_1,u)\,l(v_2,u)\,\d u\, \d v_1\, \d v_2.
\end{eqnarray*}

The Hitsuda representation  can be extended  to the class of the canonical Gaussian Volterra processes, see \cite{baud} and \cite{Sottinen}. A centered  Gaussian process $\widetilde{X}=(\widetilde{X}_t;  t\in\left[ 0,T \right])$ is equivalent in law to a Gaussian Volterra process $X$ if and only if there exits a unique centered  Gaussian process, namely $\widetilde{W}$,  satisfying \eqref{canodecom} and  \eqref{cond}, and such that 
\begin{equation}\label{eqvol}
\widetilde{X}_t=\int_0^t k(t,s)\, \d \widetilde{W}_s= X_t-\int_0^t k(t,s)\int_0^s l(s,u)\, \d W_u\, \d s,
\end{equation} 
where the kernel $k(t,s)$ and the standard Brownian motion stand for \eqref{volrep}, the canonical Volterra representation of $X$. Moreover, we have $\HH_{\widetilde{X}}(t)=\HH_{X}(t)$ for all $t$.

Under the condition \eqref{condition}, the kernel $k$ is $(\beta - \frac 12)$-homogeneous, and the centered Gaussian process  $\widetilde{X}$ is $\beta$--self-similar under $\widetilde{\P}$ since $\widetilde{W}$ is a standard Brownian motion. It is obvious that if $\widetilde{X}$ has same law as $X$, it is $\beta$--self-similar under $\P$, and this condition is also necessary, see \cite{picard}. However, in the next proposition, we will use the homogeneity property of the Volterra kernel $l$ as a necessary and  sufficient condition for the self--similarity for the process $\widetilde{X}$, and equivalently for $\widetilde{W}$, under the law $\P$.

\begin{prop}\label{proposition}
Let $X=(X_t;  t\in\left[ 0,T \right])$ be a centered  $\beta$-self--similar Gaussian process satisfying the condition \eqref{condition}, then
\begin{enumerate}
\item a centered  Gaussian process $\widetilde{X}=(\widetilde{X}_t;  t\in\left[ 0,T \right])$ is equivalent in law to $X$ if and only if $\widetilde{X}$ admits a representation of the form of
\begin{equation}
\widetilde{X}_t=X_t- t^{\beta-\frac 12}\int_0^t z(t,s)\, \d W_s,\quad 0\leq t \leq T,
\end{equation}
where $W$ is a standard Brownian motion on $[0,T]$, and the kernel $z(t,s)$ is independent of $\beta$ and expressed by
$$z(t,s)=\int_s^t F\left(\frac ut\right)\, l(u,s)\, \d u, \quad s<t,$$
for a Volterra kernel $l$ and some function $F \in L^2(\R_+, \d u)$ vanishing on $(1,\infty]$.
\item In addition, $\widetilde{X}$ is $\beta$--self-similar if and only if $l\equiv 0$.
\end{enumerate}
\end{prop}

For the proof, we need the following lemma.
\begin{lemma}\label{lemma}
If a Volterra kernel  on $[0,T]\times[0.T]$ is homogeneous with degree $(-1)$, then  it vanishes on $[0,T]\times[0.T]$. 
\end{lemma}
\begin{proof}
Let a Volterra kernel $h$ be $(-1)$-homogeneous. Combining the square integrability and the homogeneity property $h(t,s)=\frac{1}{a}\, h(\frac{t}{a}, \frac{s}{a})$, $a>0$, $0\leq s<t\leq T$, yields
$$
\int_0^T \int_0^t h(t,s)^2\, \d s\,\d t=\int_0^{\frac{T}{a}} \int_0^{\frac{t}{a}} h\left(\frac{t}{a}, \frac{s}{a}\right)^2\frac{1}{a^2}\, \d s\,\d t= \int_0^{\frac{T}{a}} \int_0^{ t'} h(t',s')^2\, \d s'\,\d t'
$$
which is finite for all $a>0$. This implies that $h$ vanishes  on $[0,T]\times[0.T]$.
\end{proof}

\begin{proof} 
(i) $X$ satisfies the condition \eqref{condition}, and by Theorem \eqref{GVP}, it admits a canonical Volterra representation with a standard Brownian motion $W$ and a kernel of the form of $k(t,s)=t^{\beta-\frac 12}F\left(\frac st\right)$, $F \in L^2(\R_+, \d u)$ vanishing on $(1,\infty]$. By using Fubini theorem,  \eqref{eqvol} gives
$$\widetilde{X}_t= X_t-\int_0^t \int_s^t k(t,u) l(u,s)\, \d u\, \d W_s, \quad 0 \leq t \leq T,$$
which proves the claim.\\
ii) Suppose that $\widetilde{X}$ is $\beta$-self--similar. From (i), $\widetilde{X}$ has the representation
$$
\widetilde{X}_t= \int_0^t \left(k(t,s)-t^{\beta-\frac 12}z(t,s)\right)\, \d W_s, \quad 0 \leq t\leq T,
$$
which is a canonical Volterra representation. Indeed, if $\LL$ denotes the Volterra integral operator associated with the Volterra kernel $l(t,s)$, the integral operator $\KK-\KK \LL=\KK(\I-\LL)$ that corresponds to the Volterra kernel $k(t,s)-t^{\beta-\frac 12}z(t,s)$ is also a Volterra integral operator, \cite{gohberg}. Here, $\I$ denotes the Identity operator. In particular, if we let $f \in L^2([0,T])$  be such that $\KK(\I-\LL)f=0$. By (i) in Remark \ref{remark}, the operator $\KK$ is injective, hence, $(\I-\LL)f=0$, i.e., $ \LL f=f$. Therefore, the Volterra integral operator $\LL$ admits an eigenvalue, which is a contradiction by (ii) in  Remark \ref{remark}. So, $f\equiv 0$.\\
Now,  using the fact that $\HH_{\widetilde{X}}(t)=\HH_{X}(t)$ for all $t$, $\widetilde{X}$  satisfies also the condition \eqref{condition}, and by Theorem \eqref{GVP},  the canonical  kernel $k(t,s)-t^{\beta-\frac 12}z(t,s)$ is ($\beta - \frac 12$)-homogeneous. For $a>0$, we write
$$k(t,s)-t^{\beta-\frac 12}z(t,s)=a^{\beta-\frac 12}\left(k\left(\frac ta,\frac sa\right)-t^{\beta-\frac 12}z\left(\frac ta,\frac sa\right)\right),$$
which implies that $z(t,s)=z(\frac ta,\frac sa)$, and by the change of variable , we have
$$\int_s^t F\left(\frac ut\right)\, l(u,s)\, \d u=\int_{\frac sa}^{\frac ta} F\left(\frac{u}{\frac ta}\right)\, l\left(u ,\frac sa\right)\, \d u = \int_s^t F\left(\frac vt\right)\, \frac 1a \,l\left(\frac va,\frac sa\right)\, \d v, \quad s<t,$$
which equivalent to
$$
\int_0^t F\left(\frac ut\right)\, l(u,s)\, \d u=\int_0^t F\left(\frac ut\right)\, \frac 1a \,l\left(\frac ua,\frac sa\right)\, \d v,\quad s<u<t.
$$
Taking derivatives with respect to $t$ on both sides, and since $F\left(\frac ut\right)\neq 0$, we obtain 
$$l(u,s)=\frac 1a \,l\left(\frac ua,\frac sa\right),\quad s<u,$$  which means that $l$ is homogeneous with degree $(-1)$. By applying the Lemma \ref{lemma}, we get $l\equiv 0$.\\
If $l\equiv 0$, we have $\E(\widetilde{X}_t\widetilde{X}_s)=\E(X_t X_s)$ which means that $\widetilde{X}\overset{d}{=}X$. Therefore, $\widetilde{X}$ is $\beta$--self-similar.
\end{proof}

\begin{remark}
 The importance of the  condition \eqref{condition} in Proposition \eqref{proposition} can been seen in the case of the fBm with index $H=1$, i.e. $B^H_t=tB^H_1$,  $0\leq t \leq T$. Here the condition \eqref{condition} fails. Since fBm is Gaussian, each process is determined by its covariance $\E(B^H_tB^H_s)=ts\, \E((B^H_1)^2)$. However, the laws of processes that correspond to different values of $\E((B^H_1)^2)$ are equivalent, on the other hand, these laws are different.  
\end{remark}




\begin{thebibliography}{10}


\bibitem{alios}
\textsc{Al\`{o}s, E., Mazet, O. and Nualart, D.}
\emph{Stochastic calculus with respect to Gaussian processes.}
 Ann. Probab. \textbf{29}, 766--801 (2001).

\bibitem{baud}
\textsc{Baudoin, F. and Nualart, D.} 
\emph{Equivalence of Volterra processes.} Stochastic Process. Appl. \textbf{107}(2), 327-350, 2003.

\bibitem{Cel}
\textsc{Jost, C.}
\emph{A note on ergodic transformations of self-similar Volterra Gaussian processes.}
 Electron. Commun. Probab.\textbf{12}, 259--266, 2007.

\bibitem{cramer}
\textsc{Cramer, H.}
\emph{On the structure of purely non-deterministic processes.} Ark. Mat. \textbf{4}, 249--266, 1961.

\bibitem{decreuse}
\textsc{Decreusefond, L. and \"{U}st\"{u}nel, A.S.}
\emph{Stochastic analysis of the fractional Brownian motion.}
Potential Anal. \textbf{10}(2), 177-214. 1999.

\bibitem{kean}
\textsc{Dym, H.and McKean, H. P.}
\emph{Gaussian processes, function theory and the inverse spectral problem.}
Academic press, New York--London, 1976.

\bibitem{gohberg}
\textsc{Gohberg, I. C. and Krein, M. G.}
\emph{Introduction to the theory of linear nonselfadjoint operators.}
Translated from the Russian by A. Feinstein. Translations of Mathematical Monographs, Vol. \textbf{18}, American Mathematical Society, Providence, R.I., 1969.

\bibitem{Hida}
\textsc{Hida, T.}
\emph{Brownian motion.}
Application of Mathematics, vol. \textbf{11}, Springer- Verlag, 1980.

\bibitem{hida}
\textsc{Hida, T.}
\emph{Canonical representations of Gaussian processes and their applications.}
Mem. Coll.Sci.Univ. Kyoto \textbf{33}, 109--155, 1960.

\bibitem{HidaHitsuda}
\textsc{Hida, T. and Hitsuda, M.}
\emph{Gaussian processes.} AMS Translations, 1993.

\bibitem{Hitsuda}
\textsc{Hitsuda, M.} 
\emph{ Representation of Gaussian processes equivalent to Wiener process.}
Osaka J. Math. \textbf{5}, 299312, 1968.

\bibitem{Karhunen}
\textsc{Karhunen, K.}
\emph{\"{U}ber die struktur station\"{a}rer zuf\"{a}lliger funktionen.} 
Ark. Mat. \textbf{1}, no. \textbf{3}, 141--160, 1950.

\bibitem{Kuo}
\textsc{Kuo, H. --H.} 
\emph{Gaussian measure in Banach spaces.}  
Springer LNM no. \textbf{463}, Springer-Verlag, Berlin, 1975.

\bibitem{Lamp}
\textsc{Lamperti, J. W.} 
\emph{Semi--stable stochastic processes.}  
Trans. Amer. Math. Soc. \textbf{104}, 62--78, 1962.

\bibitem{Lev1}
\textsc{L\'{e}vy, P.}
\emph{A special problem of Brownian motion, and a general theory of Gaussian random functions.}
Proceeding of the Third Berkeley Symposium on Math. Stat. and Prob., \textbf{2}, 133--175, 1956.

\bibitem{Lev2}
\textsc{L\'{e}vy, P.}
\emph{Sur une classe de courbes de l'espace de Hilbert et sur une \'{e}quation int\'{e}grale non lin\'{e}aire.}
Annales scientifiques de l'\'{E}.N.S. 3$^{e}$ s\'{e}rie, tome \textbf{73}, no. 2, 1956.

\bibitem{loeve}
\textsc{Lo\`{e}ve, M.}
\emph{Probability theory.} vol. II, 4th ed., Graduate Texts in Mathematics \textbf{46}, Springer-Verlag, 1978.

\bibitem{picard}
\textsc{Picard, J.}
\emph{Representation formulae for the fractional Brownian motion.} 
S\'{e}minaire de Probabilit\'{e}s, XLIII, 43:372, 2011.

\bibitem{Shiryaev}
\textsc{Shiryaev, A.}
\emph{Probability.}
Second edition. 
Graduate Texts in Mathematics, \textbf{95}. 
Springer-Verlag, New York, 1996.

\bibitem{Sottinen}
\textsc{Sottinen, T.}
\emph{On Gaussian processes equivalent in law to fractional Brownian motion.} 
Journal of Theoretical Probability \textbf{17}, no. 2, 309--325, 2004. 

\bibitem{SottinenTudor}
\textsc{Sottinen, T. and Tudor, C.A.}
\emph{On the equivalence of multiparameter Gaussian processes.}
Journal of Theoretical Probability \textbf{19}, no. 2, 461--485, 2006. 

\end{thebibliography}
\end{document}